\documentclass[11pt,a4paper]{article}

\usepackage[T1]{fontenc}
\usepackage[utf8]{inputenc}
\usepackage{authblk}
\usepackage[margin = 1in]{geometry}
\usepackage{amsmath,amssymb, amsthm}
\usepackage{color}
\usepackage{graphicx}
\usepackage{float}
\usepackage{comment}
\usepackage{url}

\newcommand{\R}{\mathbb{R}}
\newcommand{\Leb}{\mathcal{L}}
\newcommand{\C}{\mathcal{C}}

\newtheorem{thm}{Theorem}[section]
\newtheorem{rmk}[thm]{Remark}
\newtheorem{lem}[thm]{Lemma}
\newtheorem{defn}[thm]{Definition}

\title{On the Measure of the Midpoints of the Cantor Set in $\R$}
\author[1]{Enrique Alvarado\thanks{enrique.alvarado@wsu.edu}}
\author[1]{Yunfeng Hu\thanks{yunfeng.hu@wsu.edu}}
\affil[1]{Department of Mathematics and Statistics, Washington State University}

\begin{document}
\maketitle

\begin{abstract}
In this paper, we are going to discuss the following problem: Let $T$ be a fixed set in $\R^n$, and let $S$ and $B$ be two subsets in $\R^n$ such that for any $x$ in $S$, there exists an $r$ such that $x+ r T$ is a subset of $B$. How small can $B$ be if we know the size of $S$? Stein proved that for $n$ greater than or equal to 3 and $T$ is a sphere centered at origin, then $S$ having positive measure implies $B$ has positive measure by using spherical maximal operator. Later, Bourgain and Marstrand proved the similar result for $n =2$. Here we will show an example for why the result fails for $n=1$. 
\end{abstract}

\section{Introduction}

The problem in the abstract is included in the paper by Tam\'{a}s Keleti \cite{tamas-2017-unions}. The purpose of this paper is to construct a counterexample in $\mathbb{R}$ for the following theorem which holds for dimensions greater than $1$. 

\begin{thm}\label{Problem} Let $S\subset \mathbb{R}^n$ $(n\geq 2)$ be a set of positive Lebesgue measure. If $B\subset \mathbb{R}^n$ contains a sphere centered at every point of $S$, then $B$ has positive Lebesgue measure. 
\end{thm}

In 1976, Ellias Stein \cite{stein-1976-means} proved the case for when $n \geq 3$. Then in 1987 and 1986, Marstrand \cite{marstrand-1987-packing} and Bourgain \cite{bourgain-1986-averages} independently proved the theorem for dimension $n = 2$. However, the theorem is not true for $n=1.$\\

\noindent$\mathbf{Idea:}$ To construct an counterexample, the idea is to start with a set $B$ with zero length, and then construct a set $S$ with positive measure made up of all midpoints from $B$. Notice that if $S$ contains all the midpoints from $B$, then for any point $z\in S$, it will be the midpoint of a pair of points $x, y$ in $B$. And $\{x,y\}$ will be a sphere centered at $z\in S$ .\\

\begin{defn}[$1/3$-Cantor set]\label{cantorset}
Let $C_1 = [0,1]\backslash\left(\frac{1}{3}, \frac{2}{3}\right)$, and 
\begin{align*}
C_i = \frac{1}{3}C_{i-1} \cup \left(\frac{2}{3} + \frac{1}{3}C_{i-1}\right),\ \forall i\geq 2. 
\end{align*}
We define the $1/3$-Cantor set to be $\C = \bigcap_{i=1}^n C_i.$
\end{defn}

The 1/3-Cantor set $\mathcal{C}$ will take place of $B$ in the theorem, and the midpoints from $\mathcal{C}$, denoted as $M_\mathcal{C}$, will take place of $S$. We define the {\bf{midpoint set}} of $A$ in $\mathbb{R}$ as 
\begin{align*}
M_A := \left\{\frac{x + y}{2} : x, y \in A, x\neq y\right\}.
\end{align*}

Since $\C = \cap_{i=1}^\infty C_i$, to get $\mathcal{L}^1(M_\mathcal{C}) > 0$ we will show that $M_{C_i} = (0,1)$ for all $i$. After that, we then prove 

\begin{align*}
M_{\C} = M_{\cap_{n=1}^\infty C_n} = \bigcap_{n=1}^\infty M_{C_n} = (0,1).
\end{align*}

Therefore, by setting $B = \C$ and $S = M_{\C}$, we will have $\Leb(S) = 1$ while $\Leb(B) =0$, which will be our counterexample. \\

As a matter of fact, with the same construction, we may make $M_\mathcal{C}$ as large as we want by making copies of $\mathcal{C}$ in other intervals!

\section{A counterexample in $\R$}

\begin{thm}\label{existence}
There exist $S, B\subset \R$ and $B$ contains a sphere (in $\R$) centered at every point of $S$ such that $\Leb(S) >0$ but $\Leb(B)=0.$ 
\end{thm}

In subsection 2.1, we will go over a few definitions and a couple of lemmas which we will use for the proofs of Theorem \ref{existence}. We say proofs, because we will go over two distinct proofs; in subsection 2.2 we will go over an analytic induction proof and in subsection 2.3 we will provide a more geometric argument. 

\subsection{Definitions and lemmas}

\begin{rmk}
In $\R$, a sphere of radius $r$ centered at a point $x$ will be two points $x-r$ and $x+r.$
\end{rmk}

\begin{defn}[$i$th Cantor partition set for $\lbrack0,1\rbrack$]\label{indexesset}
Define 
\begin{align*}
P_i = \left\{P^k_{C_i}:  P^k_{C_i} = \left(\frac{k-1}{3^i}, \frac{k}{3^i}\right),\ k\in \{1,2,3,\cdots, 3^i\} \right\}
\end{align*}

\noindent to be the $i${\text{th}} Cantor partition set of $(0,1).$ 
\end{defn}

\begin{figure}[H]\label{1/3 Cantor Partition}
\begin{center}
\includegraphics{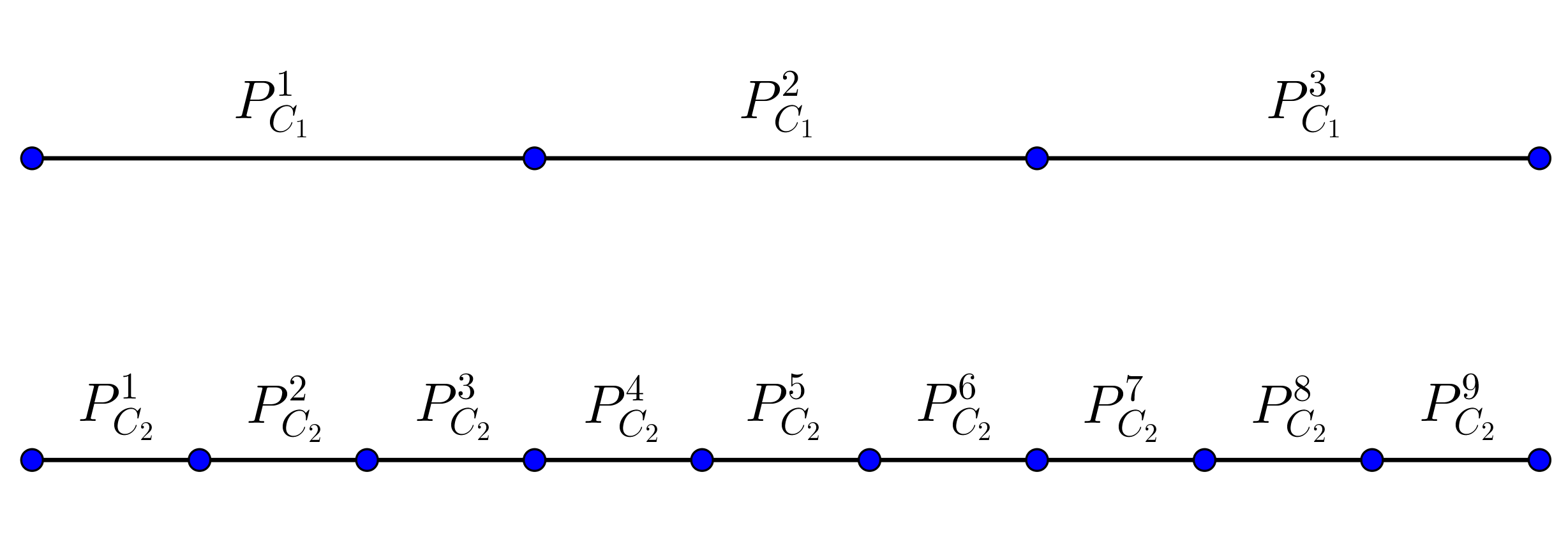}
\end{center}
\caption{1st and 2nd Cantor partition sets of $(0,1)$.}
\end{figure}

In order to prove Theorem \ref{existence}, let's first prove the following lemmas:

\begin{lem}[Average properties of indexes set]\label{indexesproperty}
If $s = \frac{1}{2}(m+n)$, then $P^s_{C_i} \subset M_{P^m_{C_i} \cup P^n_{C_i}}$.
\end{lem}

\begin{proof}
Without the loss of generality, let's assume that $m<s<n.$ Denote $x_l, x_r, y_l, y_r$ to be left and right end points for $P^m_{C_i}$ and $P^n_{C_i}$ respectively. Take any point $z\in P^s_{C_i},$ there exists a $\delta(z)$ such that 
\begin{align*}
\max\{z- x_r, y_l -z\} < \delta(z) < \min\{z - x_l, y_r -z\}.
\end{align*}
And 
\begin{align*}
x_l &= z - (z-x_l)< z - \delta(z) < z- (z-x_r) = x_r\\
y_l & = z + (y_l -z) < z + \delta(z) < z + (y_r -z) = y_r.
\end{align*}
therefore $z-\delta(z) \in P^m_{C_i}$ and $z+\delta(z)\in P^n_{C_i}$, and 

$$z = \frac{1}{2}((z-\delta(z)) + (z+\delta(z)))\in M_{P^m_{C_i} \cup P^n_{C_i}}.$$ Hence $P^s_{C_i} \subset M_{P^m_{C_i} \cup P^n_{C_i}}$.
\end{proof}

\begin{lem}[Scaling and translation property of the average set]\label{scalingproperty}
If $S = c_1 T + c_2$ ($c_1\neq 0$) and $M_T = (a,b)$, then $M_S = c_1 (a,b) + c_2.$
\end{lem}
\begin{proof}
Pick any point $z_s \in M_S$, there exist $x_s, y_s\in S$ such that $z_s = \frac{1}{2}(x_s + y_s).$ Since $S = c_1 T + c_2$, there exist $x_t, y_t\in T$ such that 
\begin{align*}
x_s = c_1 x_t + c_2,\ and\ y_s = c_1 y_t + c_2. 
\end{align*}
Therefore
\begin{align*}
z_s &= \frac{1}{2} (x_s + y_s )\\
& = \frac{1}{2}((c_1 x_t + c_2) + (c_1 y_t + c_2))\\
& = c_1 \cdot \frac{1}{2}(x_t +y_t) + c_2\\
& = c_1 z_t + c_2,
\end{align*}

\noindent where $z_t = \frac{1}{2}(x_t + y_t) \in M_T$. So $z_s \in c_1 M_T + c_2$ and then $M_S \subset c_1 M_T + c_2.$

On the other hand, pick any point $z_t \in M_T$, there exists a $x_t, y_t \in M_T$ such that $z_t = \frac{1}{2}(x_t + y_t).$ As $S = c_1 T + c_2$, $T = \frac{1}{c_1}(S- c_2).$ So there exist $x_s, y_s\in S$ such that 
\begin{align*}
x_t = \frac{1}{c_1} (x_s - c_2),\ and\ y_t =\frac{1}{c_1}(y_s - c_2).
\end{align*}
Therefore
\begin{align*}
z_t &= \frac{1}{2} (x_t + y_t) \\
& = \frac{1}{2} \left(\frac{1}{c_1} (x_s - c_2) + \frac{1}{c_1}(y_s - c_2)\right)\\
& = \frac{1}{c_1} \left(\frac{1}{2} (x_s + y_s) - c_2\right) \\
& = \frac{1}{c_1} (z_s - c_2),
\end{align*}
\noindent where $z_s = \frac{1}{2} (x_s + y_s)\in M_S.$ So $z_t \in \frac{1}{c_1} (M_S - c_2)$ and then $M_T\subset \frac{1}{c_1}(M_S - c_2)$, i.e. $c_1 M_T + c_2 \subset M_S$.  Hence $M_S = c_1 M_T + c_2.$
\end{proof}

\subsection{An analytic argument}

Now we will use the lemmas above to prove Theorem \ref{existence}.
\begin{proof}

\noindent {\bf{Claim 1:}} $M_{C_1} = (0,1).$\\
Proof of Claim 1: Let $P^1_{C_1} = \left(0, \frac{1}{3}\right)$, $P^2_{C_1} = \left(\frac{1}{3}, \frac{2}{3}\right)$ and $P^3_{C_1} = \left(\frac{2}{3}, 1\right)$ be the $1$st Cantor partition set in Definition \ref{indexesset}.
\begin{itemize}
\item[(i)] For any $z\in P^1_{C_1}$, since $P^1_{C_1}$ is open, there exists an open interval $I_z$ centered at $z$ such that $I_z \subset P^1_{C_1}$ . Therefore, $P^1_{C_1} \in M_{C_1}$. Similarly, $P^3_{C_1}\in M_{C_1}.$

\item[(ii)] For any point $z\in P^2_{C_1},$ without the loss of generality, $\frac{1}{3}< z\leq \frac{1}{2}$ (i.e. the other half $\frac{1}{2}\leq z < \frac{2}{3}$ can be solved by symmetry), there exists a $\delta(z)$ satisfying $\frac{1}{3} < \delta(z) < z$ such that 

\begin{align*}
0&< z - \delta(z) < \frac{2}{3} - \delta(z) <\frac{1}{3}\\
\frac{2}{3} &< z+ \delta(z) < z + z < 1
\end{align*} 

Therefore $z-\delta(z) \in P^1_{C_1}\subset C_1$ and $z+\delta(z) \in P^3_{C_1}\subset C_1$, and $$z = \frac{1}{2} ((z+\delta(z))+ (z-\delta(z))) \in M_{C_1}.$$ Hence $P^2_{C_1}\subset M_{C_1}.$

\item[(iii)] Since 
\begin{align*}
\frac{1}{3} = \frac{1}{2}\left(0+\frac{2}{3}\right),\ \frac{2}{3} = \left(\frac{1}{3} + 1\right),
\end{align*}
and $\{0, \frac{1}{3}, \frac{2}{3}, 1\}\in C_1$, $\{\frac{1}{3}, \frac{2}{3}\}\in M_{C_1}.$

\item[(iv)] There's no way to get $0$ or $1$ for $M_{C_1}$. 
\end{itemize}

Finally, 
\begin{align*}
M_{C_1} = (\cup_{i=1}^3 P^i_{C_1}) \cup \left\{\frac{1}{3}, \frac{2}{3}\right\} = (0,1).  
\end{align*}

\noindent {\bf{Claim 2:}} $M_{C_i} = (0,1)$ for all $i$. \\ 
Proof of Claim 2: Let $\{P^k_{C_i}\}$ be the $i$th Cantor partition set of $[0,1]$. Assume when $i=n$, $M_{C_n} = (0,1).$
\begin{itemize}
\item[(a)] $M_{C_i} \supset \left(0, \frac{1}{3}\right) \cup \left(\frac{1}{3}, \frac{2}{3}\right)$ for all $i$.\\
Since $M_{C_n} = (0,1),$ then $M_{C_n} \supset \left(0, \frac{1}{3}\right) \cup \left(\frac{1}{3}, \frac{2}{3}\right)$. When $i=n+1$, since 
\begin{align*}
C_{n+1} = \frac{1}{3}C_n \cup \left(\frac{2}{3} + \frac{1}{3} C_n\right),
\end{align*}
and by Lemma \ref{scalingproperty}, 
\begin{align*}
M_{\frac{1}{3}C_n} & = \frac{1}{3} M_{C_n} = \frac{1}{3}\left(0,1\right) = \left(0,\frac{1}{3}\right),\\
M_{\frac{2}{3}+\frac{1}{3}C_n} & = \frac{2}{3} + \frac{1}{3} M_{C_n} = \frac{2}{3} + \frac{1}{3}\left(0,1\right) = \frac{2}{3} + \left(0,\frac{1}{3}\right) = \left(\frac{2}{3},1\right),
\end{align*}

$M_{C_n+1} \supset \left(0, \frac{1}{3}\right) \cup \left(\frac{2}{3},1\right).$ Hence by induction $M_{C_i}\supset \left(0, \frac{1}{3}\right) \cup \left(\frac{1}{3}, \frac{2}{3}\right)$ for all $i$.\\


\item[(b)] $M_{C_{i}}\supset \{\frac{1}{3}, \frac{2}{3}\}$ for all $i$.\\
For each $i$, 
\begin{align*}
\frac{1}{3} = \frac{1}{2}\left(0+ \frac{2}{3}\right),\ \frac{2}{3} = \frac{1}{2}\left(\frac{1}{3} + 1\right),
\end{align*}
and $\{0, \frac{1}{3}, \frac{2}{3}, 1\} \subset C_i$, therefore $M_{C_{i}}\supset \{\frac{1}{3}, \frac{2}{3}\}$ for all $i.$

\item[(c)] $M_{C_{i}} \supset \left(\frac{1}{3},\frac{2}{3}\right)$ for all $i$.\\
Since $M_{C_n} = (0,1)$, then $M_{C_n} \supset \left(\frac{1}{3},\frac{2}{3}\right).$ For $i=n+1$, we will use the assumption $M_{C_n} = (0,1)$ and the following trick to show $M_{C_{n+1}} \supset \left(\frac{1}{3},\frac{2}{3}\right).$ \\

Take any $z\in \left(\frac{1}{3},\frac{2}{3}\right) = \left(\frac{3^n}{3^{n+1}}, \frac{2\cdot 3^{n}}{3^{n+1}}\right) = \cup_{k=3^{n}+1}^{2\cdot 3^n} P^k_{C_{n+1}},$ we want to show 
\begin{align}\label{averageincenterinterval}
z = \frac{1}{2} (a + b),
\end{align}
where $a\in (\cup_{k=1}^{3^n} P^k_{C_{n+1}}) \cap C_{n+1}$ and $b\in (\cup_{k=2\cdot 3^n +1}^{3^{n+1}} P^k_{C_{n+1}}) \cap C_{n+1}.$ Notice that since $z\in \left(\frac{1}{3}, \frac{2}{3}\right)$, there exists some $c\in \left(0,\frac{1}{3}\right),$ such that 

$$z = \frac{1}{3} +c = \frac{3^n}{3^{n+1}}+c.$$

Similarly since $C_{n+1} = \frac{1}{3}C_n \cup \left(\frac{2}{3} + \frac{1}{3}C_n\right)$, we know 

\begin{align*}
\cup_{k=2\cdot 3^n +1}^{3^{n+1}} P^k_{C_{n+1}}\cap C_{n+1} = \frac{2}{3} + (\cup_{k=1}^{3^n} P^k_{C_{n+1}}) \cap C_{n+1}.
\end{align*}

So there exists a $d\in \cup_{k=1}^{3^n} P^k_{C_{n+1}}\cap C_{n+1} \subset \left(0,\frac{1}{3}\right).$ such that $b = \frac{2}{3} + d.$
Therefore Eq. (\ref{averageincenterinterval}) is reduced to for any $c\in \left(0,\frac{1}{3}\right)$,  find $a,d\in (\cup_{k=1}^{3^n} P^k_{C_{n+1}}\cap C_{n+1})\subset  \left(0,\frac{1}{3}\right)$ such that
\begin{align*}
c = \frac{1}{2} (a+d).
\end{align*}

Note that $\cup_{k=1}^{3^n} P^k_{C_{n+1}}\cap C_{n+1} = (\cup_{k=1}^{3^{n+1}} P^k_{C_{n+1}}\cap C_{n+1}) \cap \left(0, \frac{1}{3}\right) = \frac{1}{3} C_n.$  And in (a), it is already shown that $M_{\frac{1}{3}C_n} = \left(0, \frac{1}{3}\right)$. So the existence for $a,d$ are guaranteed. Hence Eq. (\ref{averageincenterinterval}) is true and therefore $M_{C_{n+1}} \supset \left(\frac{1}{3},\frac{2}{3}\right)$ is true for all $i$.
\end{itemize}

Therefore $M_{C_i} = (0,1)$ for all $i$.\\

\noindent {\bf{Claim 3:}} $M_\C = \bigcap_{i=1}^\infty M_{C_i} = (0,1).$
\begin{itemize}
\item[(d)] One direction $M_\C \subset \bigcap_{i=1}^\infty M_{C_i} = (0,1)$ is trivial, since $\C \subset C_i$ for all $i$, then $M_\C \subset M_{C_i}$ for all $i$. Hence $M_\C \subset \bigcap_{i=1}^\infty M_{C_i} = (0,1).$

\item[(e)] For the other direction, pick any $z\in (0,1),$ that exists at least one pair $(x_n^k, y_n^k) \in C_n$ such that 
\begin{align*}
z = \frac{x_n^k + y_n^k}{2}. 
\end{align*} 

Define 
\begin{align*}
x_n^0 = \min\left\{x_n^k: z = \frac{x_n^k + y_n^k}{2}, x_n^k \in C_n\right\}.
\end{align*}

{\bf{Claim:}} $\{x_n^0\}$ converges to $x$ and $x\in \C.$

Proof of claim: First we will show $\{x^0_{n}\}$ is nondecreasing. Assume it is not true, then there exists an $m$ such that $x_{m+1}^0 \geq x_m^0.$ However, $x_{m+1}^0 \in C_{m+1} \subset C_m$, and 
\begin{align*}
z= \frac{x^0_{m+1} + y^0_{m+1}}{2}
\end{align*}
it contradicts with the construction of $x_m^0.$ Therefore $x^0_{n}$ is nondecreasing. Moreover, $0\leq x_n^0 <z$, by monotone convergence theorem, $\{x_n^0\}$ converges. Denote 
\begin{align*}
x = \lim_{m\rightarrow \infty} x_n^0.
\end{align*} 

For any $N$, $\{x_n\}_{n\geq N} \subset C_{N}$. Since $C_N$ is compact, $x^0_n\rightarrow x\in C_N$. Therefore $x \in \bigcap_{n=1}^\infty C_n = \C.$ Similarly, $\{y^0_n\}$ is nonincreasing since $y_n^0 = 2z - x_n^0$, and $z< y_n^0 <1.$ Again, by monotone convergence theorem and the same argument, $y_n^0 \rightarrow y\in \C.$ Hence $z = \frac{x+y}{2} \in M_C$. Therefore $(0,1)\subset M_C.$ 
\end{itemize}

Finally, let $B = \C$ and $S=M_C$, then for any point $x\in S$, $B$ contains a sphere around $x$, and $\mathcal{H}^1(S)=1$ but $\mathcal{H}^1(B)=0$. In fact, we can have countably many copies of $B$ such that $\mathcal{H}^1(B) = 0$ while $\mathcal{H}^1(S) = \infty.$ 
\end{proof}

\subsection{A geometric argument}
\begin{thm}
Suppose $C$, $C_i$ and $M_{C_i}$ are defined as above. Then $M_{C_i} = (0,1)$ for all integers $i \geq 1$.
\end{thm}

\begin{proof}
We will show that $M_{C_i} = (0,1)$ for all integers $i \geq 1$ by inducting on $i$. \\

\noindent$\mathbf{Base\ Case:}$ Let $i = 1$, and consider the partition set $P_i$. To show that $M_{C_1}$ contains $P_{C_1}^1$, notice that for any point $m \in P_{C_1}^1$, there exists an $r > 0$ dependent on $m$, such that $m - r$ and $m + r$ are contained in $P_{C_1}^1$. Similarly, $M_{C_1}$ contains $P_{C_1}^3$. \\

Now, to show that $P_{C_1}^2\cup\left\{\frac{1}{3}, \frac{2}{3}\right\} \subset M_{C_1}$ we have that for any point $m \in P_{C_1}^2$, $m + \frac{1}{3}$ and $m - \frac{1}{3}$ are both contained in $C_1$. Furthermore, $\frac{1}{3} = \frac{1}{2}\left(0+ \frac{2}{3}\right)$ and $\frac{2}{3} = \frac{1}{2}\left(\frac{1}{3} +1\right)$ in which case we have shown our base case.\\

\noindent$\mathbf{Induction:}$
By the inductive hypothesis, suppose the theorem holds for $i = n$ for some $n \in \mathbb{N}$. We will show that the theorem holds true for $i = n+1$. That is, we will show that $M_{C_{n+1}}$ contains $(0,1)$. \\

First consider the scaled version of $C_{n}$, i.e. $\frac{1}{3}C_{n}$. By the scaling property $\ref{scalingproperty}$, since $M_{C_{n}}$ contains $(0,1)$,  $M_{\frac{1}{3}C_{n}}$ contains $\left(0, \frac{1}{3}\right)$ and hence $M_{C_{n+1}\cap \left(0, \frac{1}{3}\right)} \subset M_{C_{n+1}}$ contains $\left(0,\frac{1}{3}\right)$. By the same scaling property, we get that that $M_{\frac{1}{3}C_{n} + \frac{2}{3}}$ contains $\left(\frac{2}{3}, 1\right)$ and hence, $M_{C_{n+1}}$ contains $\left(\frac{2}{3}, 1\right)$. \\

Lets take a quick break here to notice that $\frac{1}{3}$ and $\frac{2}{3}$ are both contained in $M_{C_{n+1}}$ because $0, \frac{1}{3}, \frac{2}{3}$, and $1$ are all in $C_{n+1}$ just as we showed in the base case. So in the following, we just need to argue that $M_{C_{n+1}}$ contains $\left(\frac{1}{3}, \frac{2}{3}\right)$.\\

Consider the partitions in $P_{n+1}\cap \left(\frac{1}{3}, \frac{2}{3}\right)$ and $P_{n}\cap (1/3, 2/3)$, which are exactly $\mathcal{P}_{n+1} = \{P_{C_{n+1}}^{3^{n} + 1}, P_{C_{n+1}}^{3^{n} + 2}, ..., P_{C_{n+1}}^{2\cdot 3^{n}}\}$, and $\mathcal{P}_{n} = \{P_{C_{n}}^{3^{n-1}+1}, ..., P_{C_{n}}^{2\cdot 3^{n-1}}\}$ respectively, see Fig \ref{partitionargument}.\\

Pick $\bar{P}_{C_{n}} \in \mathcal{P}_{n}$ and partition $\bar{P}_{C_{n}}$ into thirds to get $\{\bar{P}_{C_{n}}^1, \bar{P}_{C_{n}}^2, \bar{P}_{C_{n}}^3\}$. It is important to notice that $\bar{P}_{C_{n}}^m \in \mathcal{P}_{n+1}$ for each $m = 1, 2, 3$.\\

By the inductive hypothesis, there exists distinct partitions $P_{C_{n}}^k$ and $P_{C_{n}}^j$ contained in $C_{n}$ for which the midpoint set of $P_{C_{n}}^k\cup P_{C_{n}}^j$ contains $\bar{P}_{C_{n}}$. \\

As we did with $\bar{P}_{C_{n}}$, we may partition $P_{C_{n}}^k$ and $P_{C_{n}}^j$ into thirds, to get $\{P_{C_{n}}^{k_1}, P_{C_{n}}^{k_2}, P_{C_{n}}^{k_3}\}$ and $\{P_{C_{n}}^{j_1}, P_{C_{n}}^{j_2}, P_{C_{n}}^{j_3}\}$. However, only $\{P_{C_{n}}^{k_1}, P_{C_{n}}^{k_3}\}$ and $\{P_{C_{n}}^{j_1}, P_{C_{n}}^{j_3}\}$ are in $C_{n+1}$. 

\begin{figure}[H]
\centering
\includegraphics[width=\textwidth]{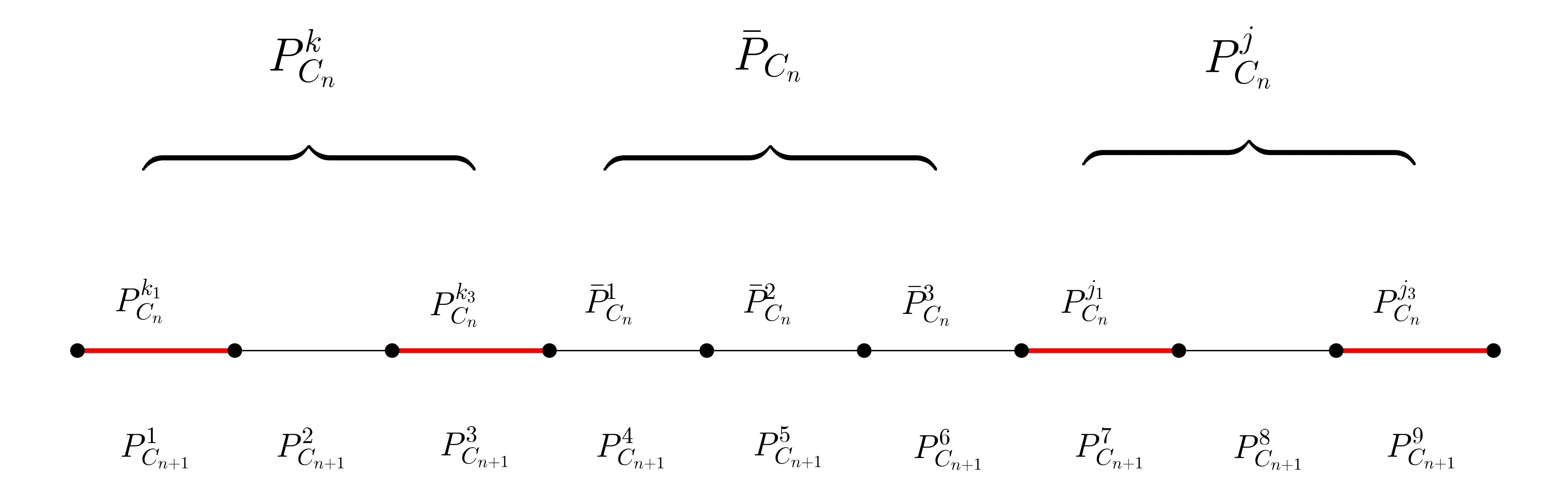}
\caption{For each $\bar{P}_{C_n}$, we go to the partitions $P_{C_n}^k$ and $P_{C_n}^j$ for which $M_{P_{C_n}^k\cup P_{C_n}^j}$ contains $\left[\frac{1}{3}, \frac{2}{3}\right]$. The first thirds of $P^k_{C_n}$ and $P_{C_n}^j$ gives us $\bar{P}^1_{C_n}$, and the last thirds of $P^k_{C_n}$ and $P_{C_n}^j$ gives us $\bar{P}^3_{C_n}$. To get the middle, $\bar{P}_{C_n}^2$, we use the first and last third of $P^k_{C_n}$ and $P_{C_n}^j$ respectively.}
\caption{Partition argument}
\label{partitionargument}
\end{figure}

Hence, we get  
\begin{align}\label{iteration}
\begin{split}
M_{P_{C_{n}}^{k_1}\cup P_{C_{n}}^{j_1}} &= \bar{P}_{C_{n}}^1\\
M_{P_{C_{n}}^{k_3}\cup P_{C_{n}}^{j_3}} &= \bar{P}_{C_{n}}^3\\
M_{P_{C_{n}}^{k_1}\cup P_{C_{n}}^{j_3}} &= \bar{P}_{C_{n}}^2.
\end{split}
\end{align}

Since this argument holds holds for arbitrary $\bar{P}_{C_{k-1}}$, we may therefore iterate through $\mathcal{P}_{n+1}$ by iterating through $\mathcal{P}_{n}$ and showing (\ref{iteration}); therefore showing that $\left[\frac{1}{3}, \frac{2}{3}\right] \subset M_{C_k}$. \\

By the principle of mathematical induction we have shown that $(0,1) \subset M_{C_i}$ for all $i \geq 1$.  
\end{proof}

\begin{thm}\label{compactness}
Suppose $\C$, $C_i$ and $M_{C_i}$ are defined as above. Then
\begin{align*}
\displaystyle\bigcap_{i=1}^\infty M_{C_i} &= M_\C.
\end{align*}
\end{thm}

\begin{proof}
To show $\cap_{i=1}^\infty M_{C_i} \subset M_\C$, we just need to notice that $\C$ being contained in all $C_i$, implies $M_\C$ is contained in all $M_{C_i}$. \\

To show that $\cap_{i=1}^\infty M_{C_i} \subset M_\C$, let $z \in \cap_{i=1}^\infty M_{C_i}$. For each $i = 1, 2, ...$ there exists $p_i := (x_i, y_i) \in C_i \times C_i$ such that $(x_i + y_i)/2 = z$. Now, for the bounded sequence $(p_i)$, there exists a subsequence $(p_{i_j})$ converging to $p := (x, y)$. To show that $p \in \C\times \C$, assume by way of contradiction that $p$ is in $\mathbb{R}^2\setminus \C\times \C$.\\

First notice that since $\C\times \C$ is closed, $\mathbb{R}^2\setminus \C\times \C$ is open, and hence there exists $\epsilon > 0$ for which the open ball $B(p, \epsilon) \subset (\mathbb{R}^2\setminus \C\times \C)$.\\

Now, since $\C\times \C = \cap_{i=1}^\infty (C_i\times C_i)$, and since $C_{i+1}\subset (C_i\times C_i)$ for all $i$, there exists $N_1 \in \mathbb{N}$ such that $\left(C_i\times C_i\right) \cap B(p, \epsilon) = \emptyset$ for all $i > N_1$. However, since $p_{i_j} \to p$ as $j\to \infty$, there exists $N_2 \in \mathbb{N}$ for which $p_{i_j} \in B(p, \epsilon)$ for all $j > N_2$. \\

Hence for any $j > \mathrm{max}\{N_1, N_2\}$, we have that $p_{i_j}$ is in both $C_{i_j}\times C_{i_j}$ and $B(p, \epsilon)$;  contradicting the fact that $(C_i\times C_i) \cap B(p, \epsilon) = \emptyset$ for all $i > N_1$. \\

We therefore get that, $p \in \C\times \C$, $(x + y)/2 = z$ and hence $z \in M_C$ thus concluding the proof.
\end{proof}

By the previous two theorems, we get that the midpoint set of the $1/3$-Cantor set is the interval $(0,1)$. We therefore get that the positive measure set, $S = M_{\mathcal{\C}} = (0,1)$ contains spheres centered about the zero measure set, $\mathcal{C}$.
\bibliographystyle{plain}
\bibliography{On_The_Measure_of_the_Midpoints_of_Cantor_Set_in_R.bib}

\end{document}